\documentclass[12pt]{amsart}
\usepackage{amsmath,amsthm,amsfonts,latexsym,amssymb,comment}
\usepackage{color}
\newcommand{\Z}{\mathbb Z}
\newcommand{\Q}{\mathbb Q}

\newcommand{\F}{\mathbb F}
\newcommand{\calC}{\mathcal C}
\newcommand{\calS}{\mathcal S}

\DeclareMathOperator{\divisor}{div}

\newtheorem{lemma}{Lemma}[section]

\newtheorem{theorem}{Theorem}

\begin{document}
\title{Squares in arithmetic progression over cubic fields}
\author{Andrew Bremner}
\address{
School of Mathematical and Statistical Sciences\\
 Arizona State University\\
 Tempe\\
 AZ 85287-1804\\
USA
}
\author{Samir Siksek}
\address{Mathematics Institute\\ 
University of Warwick\\
CV4 7AL\\
Coventry\\
UK}

\thanks{The second-named author is supported by EPSRC Programme Grant
\lq LMF: L-Functions and Modular Forms\rq\  EP/K034383/1.
}
\date{\today}

\begin{abstract}
Euler showed that there can be no more than three integer squares in arithmetic
progression. In quadratic number fields, Xarles has shown that
there can be 
arithmetic progressions of five squares, but not of six.  Here, we prove that
there are 
no cubic number fields which contain five squares in arithmetic
progression.  
\end{abstract}

\maketitle

\section{Introduction}
There are infinitely many triples of integer squares in arithmetic progression,
provided by the formulae:
\[ (m^2-2 m n-n^2)^2, \; (m^2+n^2)^2, \; (m^2+2m n-n^2)^2, \]
with common difference $4m n (m^2-n^2)$. Euler showed that it is impossible to
find four integer squares in non-trivial arithmetic progression. (Henceforth, 
by arithmetic progression, we shall always mean a non-trivial progression,
that is, one with non-zero common difference). In quadratic number fields, the 
example given by Xarles:
\[ 7^2, \; 13^2, \; 17^2, \; 409, \; 23^2 \]
shows that arithmetic progressions of squares exist with five terms.
Xarles~\cite{Xar1} shows that there are infinitely many such progressions, and 
that there cannot exist six-term arithmetic progressions of squares in quadratic fields. 
He gives examples of four squares in cubic number fields that lie in arithmetic progression, 
and implicitly shows, using a theorem of Frey, that there are only finitely many cubic fields 
that contain six squares in arithmetic progression. In this note
we prove the following result.
\begin{theorem}\label{thm:cubic}
There are no cubic number fields containing five squares in arithmetic progression.
\end{theorem}

\section{Some Geometry and Arithmetic}

The equations relating five squares in arithmetic progression are as follows:
\begin{equation}
\label{5AP}
\calS: \; a^2-2 b^2+c^2=0, \qquad b^2-2c^2+d^2=0, \qquad c^2-2d^2+e^2=0.
\end{equation}
The equations \eqref{5AP} define a curve $\calS \subset \mathbb{P}^4$ of genus 5, and $\calS$ covers the
five elliptic curves obtained by \lq\lq forgetting\rq\rq\ each variable in turn, of ranks
$0,1,0,1,0$ over $\Q$, respectively; so $\calS$ has totally split Jacobian. However, we do not use
this explicitly in what follows.

The involution of $\calS$ obtained by $(a,b,c,d,e) \rightarrow (a,-b,c,-d,e)$ 
determines a quotient curve $\calC$ of genus 3 given by the equation
\begin{equation}
\label{eqn:curvequot}
\calC: \; y^2 = x^8 + 14 x^4 +1 = (x^4+2x^3+2x^2-2x+1)(x^4-2x^3+2x^2+2x+1),
\end{equation}
via the mapping
\begin{equation}
\label{eqn:quotientmap}
(x,y)= \left( \frac{e-c}{a-c}, \; \frac{4b d(a-2c+e)^2}{(a-c)^4}  \right).
\end{equation}
Note that \eqref{eqn:quotientmap} implies
\begin{equation}
\label{squarecond}
x^4-2 x^3+2 x^2+2 x+1 = \left( \frac{2b(a-2c+e)}{(a-c)^2} \right)^2. 
\end{equation}

An arithmetic progression consisting of five squares belonging to a cubic number
field determines a
point on $\calS$ with coordinates in a cubic field (a ``cubic point")
and this maps to a cubic point on $\calC$---it
cannot map to a rational point on $\calC$ as the morphism \eqref{eqn:quotientmap}
has degree $2$.
We shall determine $J(\Q)$, showing that it is finite, where $J$ is the Jacobian of $\calC$,
and use this to determine the cubic points on $\calC$. This last step
is somewhat clearer if explained in slightly greater generality.
\begin{theorem}\label{thm:cubicgen3}
Let $\calC$ be a hyperelliptic curve of genus $3$ over $\Q$ with $P_0 \in \calC(\Q)$. 
Suppose that $J(\Q)$ is finite,
where $J$ is the Jacobian of $\calC$. Then $\calC$ has finitely many cubic
points.
\end{theorem}
\begin{proof}
For a divisor $D$ on $\calC$, we write $[D]$ for its linear
equivalence class, $L(D)$ for its Riemann--Roch space, $\ell(D)$ 
for the dimension of $L(D)$,
and $\lvert D \rvert$ for the linear system of $D$. 
Recall the closed points of $\lvert D \rvert$ correspond
to effective divisors linearly equivalent to $D$, and that
$\dim \lvert D \rvert=\ell(D)-1$.

Let $D_1,\dotsc,D_{n}$
be degree $0$ divisors such that $[D_1],\dotsc,[D_{n}]$ represent
the distinct elements of $J(\Q)$. Let $D_i^\prime=D_i+3 P_0$.
Suppose $P$ is a cubic point, and let $\Delta$ be the 
effective irreducible degree $3$
divisor obtained by taking the formal sum of conjugates of $P$.
Then $\Delta \sim D_i^\prime$ for a unique $i$, and in particular,
$\Delta \in \lvert D_i^\prime \rvert$. As the divisors
$D_i^\prime$ have degree $3$, and $\calC$ has genus $3$,
it follows from Riemann--Roch and Clifford's inequality that
$\ell(D_i^\prime)=1$ or $2$. Moreover, if $\ell(D_i^\prime)=2$ then
$D_i^\prime$ is a special divisor of degree $3$;
it is known in this case (c.f.\ \cite[Chapter 1, Exercise D.9]{ACGH}), as $\calC$ is hyperelliptic
of genus $3$, that the linear system $\lvert D_i^\prime \rvert$
contains a base point and so does not contain
irreducible divisors.

To sum up $\Delta \in \lvert D_i^\prime \rvert$ for some $1 \le i \le n$ 
such that $\ell(D_i^\prime)=1$. But 
if $\ell(D_i^\prime)=1$ then $\dim \lvert D_i^\prime \rvert=0$, and
in that case $\lvert D_i^\prime \rvert$
consists of exactly a single divisor.
\end{proof}

The proof of Theorem~\ref{thm:cubicgen3} in fact gives a succinct recipe
for determining the cubic points on a hyperelliptic genus $3$ curve,
provided that the Mordell--Weil group of its Jacobian can be computed,
and is finite. We now do this for the curve $\calC$ in \eqref{eqn:curvequot}.
\begin{lemma}\label{lem:MW}
Let $\calC$ be as in \eqref{eqn:curvequot}, and let $J$ 
be its Jacobian. Let $\infty_+$ and $\infty_-$
be the two points at infinity on the model \eqref{eqn:curvequot}
where the function $y/x^4$ respectively takes values $1$ and $-1$.
Write
\[ 
Q_1= \infty_{+}-\infty_{-}, \quad Q_2=(0,-1)-\infty_{-}, \quad Q_3=(-1,-4)-\infty_{-}. 
\]
Then
\begin{equation}\label{eqn:MW}
J(\Q) = \frac{\Z}{4\Z} [Q_1] \oplus \frac{\Z}{4\Z}[Q_2] \oplus \frac{\Z}{8\Z}[Q_3].
\end{equation}
In particular,
$\# J(\Q)=128$.
\end{lemma}
\begin{proof}
We first determine the torsion subgroup $J(\Q)_{\mathrm{tors}}$ of $J(\Q)$ 
and then show that it is equal to $J(\Q)$.
Let $Q_1$, $Q_2$, $Q_3$
be the degree $0$ divisors in the statement of the lemma; their classes in $J(\Q)$
respectively have orders $4$, $4$, $8$. Now $5$ is a prime of good reduction
for the model $\calC$, and so the reduction modulo $5$ map
\[
\pi : J(\Q)_{\mathrm{tors}} \rightarrow J(\F_5)
\]
is an injection \cite[Appendix]{Katz}. However, $\# J(\F_5) =512$,
thus $\# J(\Q)_{\mathrm{tors}}  \mid 512$.
Let $I$ be 
the subgroup of $J(\Q)$ spanned by $[Q_1]$, $[Q_2]$, $[Q_3]$.
By explicitly computing $\pi(I)$, we find that $I$ is precisely
the group on the right-hand side of~\eqref{eqn:MW}. 
Suppose $I$ is strictly contained in $J(\Q)_{\mathrm{tors}}$. 
Then there is some $P \in J(\Q)_\mathrm{tors}\backslash I$ 
such that $2P \in I$. However,
$\pi(I)/2\pi(I)$ injects into $J(\F_5)/2J(\F_5)$. Thus there is some
$Q \in I$ such that $2P=2Q$, and so $P-Q$ is
an element of order $2$.
To show that 
$J(\Q)_\mathrm{tors}=I$ it is now enough to verify that
$I$ contains all the $2$-torsion in  $J(\Q)$. We did this with the help
of Lemma~\ref{lem:2tors} below; thus $J(\Q)_\mathrm{tors}=I$. 

A computation with {\tt MAGMA}~\cite{Mag} confirms that  the $2$-Selmer
group of $J$ 
over $\Q$ is isomorphic to $(\Z/2\Z)^3$;  for this {\tt MAGMA} uses
Stoll's implementation of the $2$-descent algorithm described in \cite{Stoll}.
As $J(\Q)_{\mathrm{tors}}/2 J(\Q)_{\mathrm{tors}} \cong (\Z/2\Z)^3$
we have that the rank of $J(\Q)$ is $0$. This completes the proof.
\end{proof}

We are unable to find a description in the literature for $2$-torsion
on hyperelliptic curves of odd genus, so we give one 
here, restricting to the case where there are two rational
points at infinity.
\begin{lemma}\label{lem:2tors}
Let $K$ be a field with characteristic $\ne 2$.
Let $g\ge 3$ be an {\bf odd} integer, and
\[
f=x^{2g+2}+a_{2g+1} x^{2g+1}+\cdots+a_0 \in K[x]
\] 
be a monic separable polynomial of degree $2g+2 \equiv 0 \pmod{4}$.
Let $\calC$ be the genus $g$ hyperelliptic curve with affine model
$y^2=f(x)$
and let $J$ be its Jacobian. 
A non-trivial element of order $2$ in $J(K)$
can be represented by a divisor $D$ such that 
\begin{enumerate} 
\item[(i)] either $2D=\divisor(h)$ for some $h(x) \in K[x]$ of even degree 
$\leq (g-1)$
satisfying $h(x) \mid f(x)$, 
\item[(ii)] or $2D=\divisor(y-h_1(x))$ for some $h_1$, $h_2 \in K[x]$,
where $h_1=x^{g+1}+\frac{1}{2} a_{2g+1} x^g+\cdots$, $h_2$
has degree $\le g$ and $a \in K^*$
such that $f=h_1^2-a h_2^2$. 
\end{enumerate}
\end{lemma}
\begin{proof}
Write $\infty_+$ and $\infty_{-}$ for the usual two points at infinity.
Let $D$ be a rational degree $0$ divisor such that $[D]$ is $2$-torsion. By
Riemann--Roch we may replace $D$ by a linearly equivalent divisor
\[
D=D_0-\frac{(g-1)}{2} \infty_{+} - \frac{(g+1)}{2} \infty_{-}
\]
 where $D_0$ is positive of degree $g$.
Then $2 D_0 -(g-1) \infty_{+} -(g+1) \infty_{-}$ is a principal divisor, say 
equal to $\divisor(k)$ for some  $k \in K(C)\backslash K$.
Thus $k$ belongs to the Riemann--Roch space $L((g-1) \infty_{+} + (g+1) \infty_{-})$ which has basis
\[
1,\; x,\; x^2,\dots,\; x^{g-1},\; y-x^{g+1}-\frac{a_{2g+1}}{2} x^g.
\]
After replacing $k$ by a scalar multiple, $k$ either has the form
$k=h(x)$  where $h(x) \in K[x]$ has degree $\le g-1$,
or $k=y-h_1(x)$ with $h_1=x^{g+1}+\frac{1}{2} a_{2g+1} x^g+\cdots \in K[x]$.

Suppose first that $k=h(x)$. If $h=h_1^2 h_2$ with $\deg(h_1) \ge 1$, then
we replace $D$ by the linearly equivalent $D-\divisor(h_1)$. Thus we may
suppose $h$ is square-free. If $\alpha \in \overline{K}$ is a root of $h$ that is not
a root of $f$ then $k$ has a zero of order $1$ at 
the points $(\alpha,\pm \sqrt{f(\alpha)})$, which is
impossible. We see that $h(x) \mid f(x)$. Finally, $h$ must have even degree
so that the valuations at the poles at infinity are even.

Next suppose that $k=y-h_1(x)$. The divisor of zeros of $k$ has the form
 $2 D_1$ where $D_1$ is positive. The divisor of zeros of $k^\prime=y+h(x)$
has the form $2 D_1^\prime$ where $D_1^\prime$ is the image of
$D_1$ under the hyperelliptic involution. 
Now if $\alpha \in \overline{K}$
is a common root of $h_1$ and $f$ then $k=y-h_1(x)$ has valuation $1$ at
$(\alpha,0)$ which is impossible. In particular, $D_1$, $D_1^\prime$
do not contain Weierstrass points.
It is easy to deduce that
\[
f(x)-h_1(x)^2=(y+h_1(x))(y-h_1(x))
\]
has the form $a h_2(x)^2$ for some $a \in K^*$ and $h_2 \in K[x]$.
Clearly the degree of $h_2$ is at most $g$.
\end{proof}



\section{Proof of Theorem~\ref{thm:cubic}}

Let $\calS$ and $\calC$ be as in \eqref{5AP} and \eqref{eqn:curvequot}.
It is sufficient to determine the cubic points on $\calC$ and verify
that none of these pull back to cubic points on $\calS$.
Let $P$ be a cubic point on $\calC$ and let $\Delta$ be the
irreducible effective degree $3$ divisor obtained by taking the formal
sum of the conjugates of $P$. To determine the possibilities for $\Delta$
we follow the strategy of the proof of Theorem~\ref{thm:cubicgen3}.
Let $D_1,\dotsc,D_{128}$ be degree $0$ divisors on $\calC$ representing the $128$
elements of $J(\Q)$ (by Lemma~\ref{lem:MW}, these
can be expressed as linear combinations of $Q_1$, $Q_2$, $Q_3$).
Let $D_i^\prime=D_i+3 \infty_-$. By the final paragraph of the 
proof of Theorem~\ref{thm:cubicgen3} we know that $\Delta$ is the unique rational
divisor in $\lvert D_i^\prime \rvert$ for some $i$ such that $\ell(D_i^\prime)=1$.
We find that there are precisely $120$ values of $i$ for which $\ell(D_i^\prime)=1$;
if $f_i$ is a generator for $L(D_i^\prime)$ then the unique rational
divisor in $\lvert D_i^\prime \rvert$ is given by $D_i^\prime+\divisor(f_i)$.
It turns out only $16$ of the $120$ effective divisors $D_i^\prime+\divisor(f_i)$
are irreducible. This yields $16$ cubic points on $\calC$ up to conjugation,
and we verify that none of these pull back to cubic points on $\calS$.
As an example,
$(\theta, 2\theta^2+\theta-1)$ is a cubic point on $\calC$, where $\theta^3-2\theta^2+2\theta+1=0$.
This point
pulls back to
\[(a,b,c,d,e)=(-\phi^4+2\phi^2+2, \; \phi, \; \phi^4-2\phi^2, \; 2\phi^5-4\phi^3-3\phi, \; -\phi^4+4\phi^2), \]
where 
\[ \phi^6 - 2\phi^4 - 2\phi^2 - 1 = 0  \qquad (\phi^2=-1/\theta).  \]

\bigskip

\noindent{\textbf{Remark.}} 
A similar argument determines all rational points on $\calC$, and all points on $\calC$ that are 
defined over quadratic extensions of $\Q$. This latter in particular yields an alternative proof of
a result of Gonzal\'ez-Jim\'enez \& Xarles related to five squares in arithmetic progression in quadratic fields. 
We sketch details, as follows.  

Suppose $P \in \calC(\Q)$. Then $P-\infty_{-} \sim D_i$ for a unique $i$, so that $P \sim D_{i,1}$ where
we denote by $D_{i,j}$ the divisor $D_i+j \cdot \infty_{-}$. Thus $P \in |D_{i,1}|$, a space of
dimension $\ell(D_{i,1})-1$.  Now $\ell=\ell(D_{i,1})$ takes values $0,1$, with $|D_{i,1}|$ empty when 
$\ell=0$ (in $120$ instances). For $\ell=1$ (occurring in $8$ instances), $|D_{i,1}|$ is of dimension $0$,
and comprises a single divisor, namely $D_{i,1}+(f)$, where $f$ is a generator for $L(D_{i,1})$. 
The eight instances where $\ell=1$ now correspond to the eight rational points $(1,\pm 1, 0)$,
$(0, \pm 1, 1)$, and $(\pm 1, \pm 4)$ on $\calC$. 

Suppose now $P$ is a quadratic point, and let $\Delta$ be the effective irreducible divisor of degree $2$ 
obtained by taking the formal sum of $P$ with its conjugate. Then $\Delta \sim D_{i,2}$ for a unique $i$, 
and $\Delta \in |D_{i,2}|$.  For $1 \leq i \leq 128$, $\ell=\ell(D_{i,2})$ takes the values $0,1,2$. 
If $\ell=0$ (in $93$ cases), then $|D_{i,2}|$ is empty.  If $\ell=1$ (in $34$ cases), then $\dim|D_{i,2}|=0$, 
and the linear system comprises the single divisor $D_{i,2}+(f) (= \Delta)$ where $f$ generates $L(D_{i,2})$. 
In all but two of these instances, $\Delta$ is reducible; the cases where $\Delta$ is irreducible are when
$D_i=3 Q_2+4 Q_3$ and $D_i=2 Q_1+Q_2+4 Q_3$, corresponding to $\Delta=(i,4)+(-i,4)$, and $(i,-4)+(-i,-4)$. 
And when $\ell=2$, which occurs precisely when $D_i=Q_1$, then $L(D_{i,2})$ has basis$\{1,x\}$, and the 
linear space $|D_{i,2}|$ comprises divisors of type $(t,\sqrt{D})+(t,-\sqrt{D})$ for $t \in \Q$, 
$D$ not a rational square.  In conclusion, $\Delta$ is thus either fixed under the hyperelliptic involution, 
or is one of the pairs $(\pm i,4)+(\mp i,4)$.

The pullback to $\calS$ of the point $(i,4)$ is the point 
\[
(a,b,c,d,e)=(\sqrt{2}, 1, 0, i, i\sqrt{2}),
\]
so does not provide a quadratic point of $\calS$. The pullback of the point $(t,\sqrt{D})$ for $t \in \Q$
is the point $(a,b,c,d,e)=$
\[
(t^2-2t-1, \sqrt{t^4-2t^3+2t^2+2t+1}, t^2+1, 
\sqrt{t^4+2t^3+2t^2-2t+1}, t^2+2t-1),
\]
where $D=t^8+14t^4+1= (t^4-2t^3+2t^2+2t+1)(t^4+2t^3+2t^2-2t+1)$. For the point to be defined
over $\Q(\sqrt{D})$ it is necessary and sufficient that precisely one of $t^4 \mp 2t^3+2t^2 \pm 2t+1$ 
be a rational square, the other a non-square.
This recovers Lemma 4 of Gonzal\'ez-Jim\'enez \& Xarles~\cite{GJX}, that an arithmetic progression of five squares 
in a quadratic field is without loss of generality of the form $(x_0^2,x_1^2,x_2^2,D x_3^2,x_4^2)$ 
for rational integers $x_i$ and $D$.

\end{document}